\documentclass[twoside]{amsart}

\usepackage{amsmath,amsfonts,amsthm,mathrsfs}
\usepackage{amssymb}

\usepackage[unicode,bookmarks]{hyperref} %ocgcolorlinks
\usepackage[usenames,dvipsnames]{xcolor}
\hypersetup{colorlinks=false}%,citecolor=NavyBlue,linkcolor=BrickRed,urlcolor=Green} %Orange

\usepackage{expl3}

\ExplSyntaxOn
\prop_new:N \g_cite_map_prop
\tl_new:N \l_citekey_result_tl

\cs_new:Npn \mapcitekey #1#2 {
  \clist_map_inline:nn {#2}
       {  \prop_gput:Nnn  \g_cite_map_prop  {##1} {#1}   }
}

\cs_new:Npn \getcitekey #1 {
   \prop_get:NoN \g_cite_map_prop{#1}  \l_citekey_result_tl
   \quark_if_no_value:NF \l_citekey_result_tl
       {  \tl_set_eq:NN #1  \l_citekey_result_tl  }
}

\cs_new:Npn \showcitekeymaps {\prop_show:N  \g_cite_map_prop }
\ExplSyntaxOff

\usepackage{etoolbox}
\makeatletter
\patchcmd{\@citex}{\if@filesw}{\getcitekey\@citeb \if@filesw}%
    {\typeout{*** SUCCESS ***}}{\typeout{*** FAIL ***}}
\patchcmd{\nocite}{\if@filesw}{\getcitekey\@citeb \if@filesw}%
    {\typeout{*** SUCCESS ***}}{\typeout{*** FAIL ***}}
\makeatother

\usepackage{enumitem}

\usepackage{chngcntr}

\usepackage{tikz}
\usepackage{tikz-cd}
\tikzset{commutative diagrams/arrow style=Latin Modern}

\newcommand{\abs}[1]{\lvert #1 \rvert}

\newcommand{\tensor}{\otimes}

\newcommand{\NN}{\mathbb{N}}

\newcommand{\CC}{\mathbb{C}}

\newcommand{\PP}{\mathbb{P}}

\newcommand{\shf}[1]{\mathscr{#1}}
\newcommand{\OX}{\shf{O}_X}

\newcommand{\restr}[1]{\big\vert_{#1}}

\def\overbar#1#2#3{{%
	\setbox0=\hbox{\displaystyle{#1}}%
	\dimen0=\wd0
	\advance\dimen0 by -#2 
	\vbox {\nointerlineskip \moveright #3 \vbox{\hrule height 0.3pt width \dimen0}%
		\nointerlineskip \vskip 1.5pt \box0}%
}}

\newcommand{\fu}{f^{\ast}}
\newcommand{\fl}{f_{\ast}}

\newcommand{\pu}{p^{\ast}}

\newcommand{\tl}{t_{\ast}}
\newcommand{\gu}{g^{\ast}}

\newcommand{\muu}{\mu^{\ast}}

\newcommand{\shF}{\shf{F}}

\newcommand{\shO}{\shf{O}}

\makeatletter
\let\@@seccntformat\@seccntformat
\renewcommand*{\@seccntformat}[1]{%
  \expandafter\ifx\csname @seccntformat@#1\endcsname\relax
    \expandafter\@@seccntformat
  \else
    \expandafter
      \csname @seccntformat@#1\expandafter\endcsname
  \fi
    {#1}%
}
\newcommand*{\@seccntformat@subsection}[1]{%
  \textbf{\csname the#1\endcsname.}
}
\makeatother

\makeatletter
\let\@paragraph\paragraph
\renewcommand*{\paragraph}[1]{%
	\vspace{0.3\baselineskip}%
	\@paragraph{\textit{#1}}%
}
\makeatother

\counterwithin{equation}{subsection}
\counterwithout{subsection}{section}
\counterwithin{figure}{subsection}

\newtheorem{theorem}[equation]{Theorem}
\newtheorem*{theorem*}{Theorem}
\newtheorem{lemma}[equation]{Lemma}
\newtheorem*{lemma*}{Lemma}

\newtheorem*{corollary*}{Corollary}
\newtheorem{proposition}[equation]{Proposition}
\newtheorem*{proposition*}{Proposition}
\newtheorem{conjecture}[equation]{Conjecture}
\newtheorem*{conjecture*}{Conjecture}

\theoremstyle{definition}

\newtheorem*{definition*}{Definition}
\theoremstyle{remark}

\newtheorem{example}[equation]{Example}
\newtheorem*{example*}{Example}
\newtheorem*{problem*}{Problem}

\theoremstyle{plain}

\newcommand{\theoremref}[1]{\hyperref[#1]{Theorem~\ref*{#1}}}
\newcommand{\lemmaref}[1]{\hyperref[#1]{Lemma~\ref*{#1}}}
\newcommand{\definitionref}[1]{\hyperref[#1]{Definition~\ref*{#1}}}
\newcommand{\propositionref}[1]{\hyperref[#1]{Proposition~\ref*{#1}}}
\newcommand{\conjectureref}[1]{\hyperref[#1]{Conjecture~\ref*{#1}}}
\newcommand{\corollaryref}[1]{\hyperref[#1]{Corollary~\ref*{#1}}}
\newcommand{\exampleref}[1]{\hyperref[#1]{Example~\ref*{#1}}}
\newcommand{\exerciseref}[1]{\hyperref[#1]{Exercise~\ref*{#1}}}

\makeatletter
\let\old@caption\caption
\renewcommand*{\caption}[1]{%
	\setcounter{figure}{\value{equation}}%
	\stepcounter{equation}%
	\old@caption{#1}\relax%
}
\makeatother

\newcounter{intro}

\newtheorem{intro-conjecture}[intro]{Conjecture}
\newtheorem{intro-corollary}[intro]{Corollary}
\newtheorem{intro-theorem}[intro]{Theorem}

\newcommand{\OY}{\shO_Y}

\newcommand{\newpar}{\subsection{}}
\newcommand{\parref}[1]{\hyperref[#1]{\S\ref*{#1}}}
\newcommand{\chapref}[1]{\hyperref[#1]{Chapter~\ref*{#1}}}

\makeatletter
\newcommand*\if@single[3]{%
  \setbox0\hbox{${\mathaccent"0362{#1}}^H$}%
  \setbox2\hbox{${\mathaccent"0362{\kern0pt#1}}^H$}%
  \ifdim\ht0=\ht2 #3\else #2\fi
  }
\newcommand*\rel@kern[1]{\kern#1\dimexpr\macc@kerna}
\newcommand*\widebar[1]{\@ifnextchar^{{\wide@bar{#1}{0}}}{\wide@bar{#1}{1}}}
\newcommand*\wide@bar[2]{\if@single{#1}{\wide@bar@{#1}{#2}{1}}{\wide@bar@{#1}{#2}{2}}}
\newcommand*\wide@bar@[3]{%
  \begingroup
  \def\mathaccent##1##2{%
    \if#32 \let\macc@nucleus\first@char \fi
    \setbox\z@\hbox{$\macc@style{\macc@nucleus}_{}$}%
    \setbox\tw@\hbox{$\macc@style{\macc@nucleus}{}_{}$}%
    \dimen@\wd\tw@
    \advance\dimen@-\wd\z@
    \divide\dimen@ 3
    \@tempdima\wd\tw@
    \advance\@tempdima-\scriptspace
    \divide\@tempdima 10
    \advance\dimen@-\@tempdima
    \ifdim\dimen@>\z@ \dimen@0pt\fi
    \rel@kern{0.6}\kern-\dimen@
    \if#31
      \overline{\rel@kern{-0.6}\kern\dimen@\macc@nucleus\rel@kern{0.4}\kern\dimen@}%
      \advance\dimen@0.4\dimexpr\macc@kerna
      \let\final@kern#2%
      \ifdim\dimen@<\z@ \let\final@kern1\fi
      \if\final@kern1 \kern-\dimen@\fi
    \else
      \overline{\rel@kern{-0.6}\kern\dimen@#1}%
    \fi
  }%
  \macc@depth\@ne
  \let\math@bgroup\@empty \let\math@egroup\macc@set@skewchar
  \mathsurround\z@ \frozen@everymath{\mathgroup\macc@group\relax}%
  \macc@set@skewchar\relax
  \let\mathaccentV\macc@nested@a
  \if#31
    \macc@nested@a\relax111{#1}%
  \else
    \def\gobble@till@marker##1\endmarker{}%
    \futurelet\first@char\gobble@till@marker#1\endmarker
    \ifcat\noexpand\first@char A\else
      \def\first@char{}%
    \fi
    \macc@nested@a\relax111{\first@char}%
  \fi
  \endgroup
}
\makeatother

\mapcitekey{Campana+Peternell:GeometricStability}{CP}
\mapcitekey{Mori:Bowdoin}{Mori}
\mapcitekey{Fujino+Matsumura:InejectivityTheorems}{FM}
\mapcitekey{Demailly+Hacon+Paun:ExtensionTheorems}{DHP}
\mapcitekey{Paun+Takayama:Positivity}{PT}
\mapcitekey{Hashizume:NonVanishing}{H}
\mapcitekey{Graber+Harris+Starr:Families}{GHS}
\mapcitekey{Kollar+Miyaoka+Mori:RationallyConnected}{KMM}
\mapcitekey{Kollar:RationalCurves}{Kollar}
\mapcitekey{Lazarsfeld:Positivity2}{Lazarsfeld}
\mapcitekey{Demailly:MultiplierIdealSheaves}{Demailly}
\mapcitekey{Boucksom+Demailly+Paun+Peternell:Pseudo-effective}{BDPP}

\newcommand{\shI}{\mathcal{I}}

\begin{document}

%========================================================
\title{Singular metrics and a conjecture by Campana and Peternell}

\author[Ch.~Schnell]{Christian Schnell}
\address{Department of Mathematics, Stony Brook University, Stony Brook, NY 11794-3651}
\email{christian.schnell@stonybrook.edu}

\begin{abstract}
	A conjecture by Campana and Peternell says that if a positive multiple
	of $K_X$ is linearly equivalent to an effective divisor $D$ plus a
	pseudo-effective divisor, then the Kodaira dimension of $X$ should be at
	least as big as the Iitaka dimension of $D$. This is a very useful
	generalization of the non-vanishing conjecture (which is the case $D=0$). We use
	recent work about singular metrics on pluri-adjoint bundles to show
	that the Campana-Peternell conjecture is almost equivalent to the non-vanishing
	conjecture.
\end{abstract}
\date{\today}
\maketitle
%========================================================

\section{Introduction}

\newpar
Let $X$ be a smooth projective variety over the complex numbers, and denote by $K_X$
the canonical divisor class. The non-vanishing conjecture predicts that if $K_X$ is
pseudo-effective, then some positive multiple of $K_X$ is
effective.\footnote{Hashizume \cite{H} has shown that this version is equivalent to
the full non-vanishing conjecture for lc pairs.}
This is a special case of the famous abundance conjecture from the minimal model
program. At some point, Campana and Peternell \cite[Conj.~on p.~43]{CP} suggested
the following generalization of the non-vanishing conjecture.

\begin{conjecture}[Campana-Peternell] \label{conj:CP}
Let $X$ be a smooth projective variety, and $D$ an effective divisor on $X$. Suppose that,
for some $m \geq 1$, the divisor class $m K_X - D$ is pseudo-effective. Then one
should have $\kappa(X) \geq \kappa(X, D)$.
\end{conjecture}

Here $\kappa(X, D)$ is the Iitaka dimension of the divisor $D$, and $\kappa(X) =
\kappa(X, K_X)$ the Kodaira dimension of $X$. The Campana-Peternell conjecture
contains the non-vanishing conjecture as the special case $D = 0$, and therefore sits
somewhere in between the non-vanishing conjecture and the abundance conjecture.
It is obviously weaker than the full abundance conjecture, but still extremely useful
in practice, for example for questions related to the behavior of Kodaira dimension
in algebraic fiber spaces \cite{families}.

\newpar
The purpose of this note is to show that the Campana-Peternell conjecture is
almost equivalent to the non-vanishing conjecture. (I believe that the two
conjectures are in fact equivalent; I write ``almost'' because there is one special
case involving rationally connected varieties where the question remains unresolved.)
The main tool is the existence of certain singular
hermitian metrics on pluri-adjoint bundles, introduced by P\u{a}un and Takayama
\cite{PT}. For a non-expert like myself, one mysterious aspect of the non-vanishing
conjecture -- and of \conjectureref{conj:CP} more generally -- is how adding
multiples of $K_X$ can possibly make things better. We will see below that the metric
techniques provide at least one mechanism for this. 

\newpar
I thank Christopher Hacon for answering some questions about the abundance
conjecture. I also had a useful email exchange about the Campana-Peternell
conjecture with Behrouz Taji. While writing this paper, I have been partially supported by NSF grant DMS-1551677 and a Simons Fellowship from the Simons Foundation. I thank both of
these institutions for their support, and the Max-Planck-Institute for providing me
with excellent working conditions during my stay in Bonn. 

\section{Equivalent formulations of the conjecture}

\newpar \label{par:algorithm}
We begin by formulating \conjectureref{conj:CP} in a way that is closer to the
non-vanishing conjecture. Campana and Peternell observed that, up to birational
modifications of $X$, one can always assume that the divisor $D$ is base-point free.
Here is why. Choose $n$ large enough, in order for the linear system $\abs{nD}$ to
give the Iitaka fibration of $D$, and let $\mu \colon X' \to X$ be an embedded
resolution of singularities of the linear system $\abs{nD}$. Then one has a
decomposition $\muu \abs{nD} = \abs{F} + E$, where $F$ is base-point free and $E$ is
effective. But now 
\[
	nm K_{X'} - F \equiv n \muu(m K_X - D) + nm K_{X'/X} + E
\]
is the sum of a pseudo-effective divisor and an effective divisor, hence
pseudo-effective. This reduces \conjectureref{conj:CP} to the case where $D$ is
base-point free.

\newpar
Now suppose that $D$ is base-point free. After replacing $D$ by a multiple, we may
assume that the linear system $\abs{D}$ defines a surjective morphism $f \colon X \to
Y$ to a projective variety $Y$ with $\dim Y = \kappa(X, D)$. We also get a very ample
divisor $H$ such that $\fu H = D$; consequently, $mK_X - \fu H$ is pseudo-effective.
After replacing $f \colon X \to Y$ by its Stein factorization, we may assume that $Y$
is normal, and that $f$ has connected fibers.  Let $\nu \colon Y' \to Y$ be a
resolution of singularities, and choose a compatible resolution of singularities $\mu
\colon X' \to X$, giving us a commutative diagram
\[
	\begin{tikzcd}
		X' \dar{f'} \rar{\mu} & X \dar{f} \\
		Y' \rar{\nu} & Y.
	\end{tikzcd}
\]
The divisor $\nu^{\ast} H$ is big and nef; after replacing $H$ by a big enough
multiple, we can therefore find an ample divisor $H'$ on $Y'$ such that $\nu^{\ast} H
- H'$ is effective. Now
\[
	m K_{X'} - (f')^{\ast} H' \equiv m K_{X'/X} + \mu^{\ast} \bigl( m K_X - \fu H \bigr)
	+ (f')^{\ast} \bigl( \nu^{\ast} H - H' \bigr)
\]
is still pseudo-effective, and since $\kappa(X') = \kappa(X)$,
\conjectureref{conj:CP} is reduced to the case where $D$ is the pullback of an ample
divisor along an algebraic fiber space.

\newpar
This leads to an equivalent formulation of
\conjectureref{conj:CP}. Let $f \colon X \to Y$ be an algebraic fiber space; by this
I mean that $X$ and $Y$ are smooth projective varieties, and that $f$ is surjective
with connected fibers. Let $H$ be an ample divisor on $Y$, and suppose that for some
$m \geq 1$, the divisor class $m K_X - \fu H$ is pseudo-effective. The
Campana-Peternell conjecture is equivalent to the statement that $\kappa(X) \geq \dim
Y$; this is clear from the discussion above.

\newpar \label{par:strong}
In fact, one can get a even stronger prediction out of \conjectureref{conj:CP}.
To state it, let $F$ denote the general fiber of the algebraic fiber space $f \colon
X \to Y$.  If the divisor class $m K_X - \fu H$ is pseudo-effective, then $K_F$ is
also pseudo-effective; according to the non-vanishing conjecture, this should imply
that $\kappa(F) \geq 0$. 

\begin{lemma} \label{lem:fiber-space}
	Let $f \colon X \to Y$ be an algebraic fiber space with general fiber $F$.
	Let $H$ be an ample divisor on $Y$, and suppose that $m K_X - \fu H$ is
	pseudo-effective for some $m \geq 1$. Then \conjectureref{conj:CP} predicts that 
	\begin{equation} \label{eq:conj-strong}
		\kappa(X) = \kappa(F) + \dim Y.
	\end{equation}
\end{lemma}

\begin{proof}
	Since $K_F$ is pseudo-effective, \conjectureref{conj:CP}, applied to the smooth
	projective variety $F$, predicts
	that $\kappa(F) \geq 0$. Choose an integer $r \geq 1$ such that $r K_F$ is
	effective. Then $\fl \OX(r
K_X)$ will be nonzero, and so $\fl \OX(r K_X) \tensor \OY(\ell H)$ will have sections
for $\ell \gg 0$. In other words, the divisor class $r K_X + \fu(\ell H)$ is effective. By
\cite[Prop.~1.14]{Mori}, we have
\[
	\kappa \bigl( X, rK_X + \fu(\ell H) \bigr) = \kappa(F) + \dim Y.
\]
The reason is that $rK_X + \fu(\ell H) - \fu H$ is effective once $\ell \gg
0$; and that the restriction of $rK_X + \fu(\ell H)$ to the general fiber $F$ is
equal to $rK_F$. But now
\[
	(m \ell + r) K_X - \bigl( r K_X + \fu(\ell H) \bigr)
		= \ell (m K_X - \fu H) 
\]
is of course still pseudo-effective, and so \conjectureref{conj:CP} is also claiming
that
\[
	\kappa(X) \geq \kappa \bigl( X, r K_X + \fu(\ell H) \bigr) 
			= \kappa(F) + \dim Y.
\]
On the other hand, one always has $\kappa(X) \leq \kappa(F) + \dim Y$ by the easy
addition formula \cite[Cor.~2.3]{Mori}, and hence \eqref{eq:conj-strong}.
\end{proof}

\newpar \label{par:Mori}
According to \cite[Prop.~1.14]{Mori}, the identity in \eqref{eq:conj-strong} is
equivalent to saying that $m K_X - \fu H$ becomes effective for $m$ sufficiently
large and divisible. The Campana-Peternell conjecture therefore looks exactly like
non-vanishing even in the general case: if $m K_X - \fu H$ is pseudo-effective
for some $m \geq 1$, then $m K_X - \fu H$ should actually be effective for $m$
sufficiently large and divisible.

\newpar
Still assuming that the non-vanishing conjecture is true, we can further reduce the problem
to algebraic fiber spaces with $\kappa(F) = 0$. The argument runs as follows. Suppose
that we have an algebraic fiber space $f \colon X \to Y$ with general fiber $F$, 
and an ample line bundle $H$ on $Y$, such that $m K_X - \fu H$ is
pseudo-effective for some $m \geq 1$.  As in the proof of \lemmaref{lem:fiber-space}, 
\[
	(m \ell + r) K_X - \bigl( r K_X + \fu (\ell H) \bigr)
\]
is pseudo-effective for suitably chosen $r,\ell \in \NN$. After another application
of the procedure in \parref{par:algorithm}, this time for the divisor class $rK_X +
\fu(\ell H)$, we get a new algebraic fiber space $f' \colon X' \to Y'$, with $X'$
birational to $X$, such that $\dim Y' = \kappa(F) + \dim Y$. We also get a new ample
divisor $H'$ on $Y'$ such that $m' K_{X'} - f'^{\ast} H'$ is pseudo-effective for
some $m' \geq 1$.  We can now iterate this procedure, and since 
\[
	\dim Y' = \dim Y + \kappa(F) \geq \dim Y, 
\]
we must eventually arrive at an algebraic fiber space $f^{(n)} \colon X^{(n)} \to
Y^{(n)}$ whose general fiber $F^{(n)}$ has Kodaira dimension $0$.

\newpar
In this manner, we reduce the proof of the Campana-Peternell conjecture, modulo the
non-vanishing conjecture, to the following special case:

\begin{conjecture} \label{conj:CP-strong}
	Let $f \colon X \to Y$ be an algebraic fiber space with $\kappa(F) \geq 0$.
	Let $H$ be an ample divisor on $Y$. If $m K_X - \fu H$ is
	pseudo-effective for some $m \geq 1$, then $mK_X - \fu H$ becomes effective for
	$m$ sufficiently large and divisible.
\end{conjecture}

We have seen above that \conjectureref{conj:CP-strong}, together with the
non-vanishing conjecture, is equivalent to \conjectureref{conj:CP}. Of course, I have
no idea of how to prove the non-vanishing conjecture; my point is that
\conjectureref{conj:CP-strong} is the part of the Campana-Peternell conjecture that
one can hope to solve using existing methods.

\newpar
As the discussion in \parref{par:strong} shows, the Campana-Peternell conjecture
also has implications for Iitaka's conjecture on subadditivity of the Kodaira
dimension. Let $f \colon X \to Y$ be an algebraic fiber space with general fiber $F$.
The Campana-Peternell conjecture reduces the problem of proving the inequality 
\[
	\kappa(X) \geq \kappa(F) + \kappa(Y)
\]
for arbitrary algebraic fiber spaces to the special case $\kappa(Y) = 0$. 

\begin{proposition} \label{prop:Iitaka}
Suppose that \conjectureref{conj:CP-strong} is true. If the inequality
\[
	\kappa(X) \geq \kappa(F) + \kappa(Y) 
\]
holds for every algebraic fiber space $f \colon X \to Y$ with $\kappa(Y) = 0$,
then it holds for every algebraic fiber space.
\end{proposition}

\begin{proof}
We can obviously assume that $\kappa(F) \geq 0$ and $\kappa(Y) \geq 0$, because the
statement is vacuous otherwise. After a birational modification, we can assume that
the Iitaka fibration of $Y$ is a morphism $p \colon Y \to Z$, and that $m K_Y - \pu
H$ is effective for some $m \geq 1$ and some ample divisor $H$ on $Z$. Let $P$ be
the general fiber of $p$; by construction, one has $\kappa(P) = 0$. Set $g = p \circ
f$, and denote by $G$ the general fiber of the resulting algebraic fiber space $g
\colon X \to Z$.
\[
\begin{tikzcd}
X \arrow[bend right=40]{rr}{g} \rar{f} & Y \rar{p} & Z
\end{tikzcd}
\]
Since $\kappa(F) \geq 0$, the relative canonical class $K_{X/Y}$ is pseudo-effective;
consequently, 
\[
	m K_X - \gu H \equiv m K_{X/Y} + \fu(m K_Y - \pu H)
\]
is also pseudo-effective. But then \eqref{eq:conj-strong} gives $\kappa(X) = \kappa(G) +
\dim Z = \kappa(G) + \kappa(Y)$. The conclusion is that if one knew subadditivity for
the algebraic fiber space $f \colon G \to P$, whose general fiber is $F$ and whose base $P$
satisfies $\kappa(P) = 0$, then one would get $\kappa(G) \geq \kappa(F)$, and
therefore the desired subadditivity for $f \colon X \to Y$.
\end{proof}

\section{The main result}

\newpar
We are going to show that \conjectureref{conj:CP-strong} is true under the additional
(but hopefully unnecessary) assumption  that the canonical divisor $K_Y$ is
pseudo-effective.

\begin{theorem} \label{thm:main}
	Let $f \colon X \to Y$ be an algebraic fiber space with $\kappa(F) \geq 0$.
	Let $H$ be an ample divisor on $Y$. If $m K_X - \fu H$ is
	pseudo-effective for some $m \geq 1$, then $mK_X - \fu H$ becomes effective for
	$m$ sufficiently large and divisible, provided that $K_Y$ is pseudo-effective.
\end{theorem}

An equivalent formulation is that, under the assumptions of the theorem, 
\[
	\kappa(X) = \kappa(F) + \dim Y.
\]
This is clear from the discussion in \parref{par:Mori}.

\newpar
The strategy of the proof is as follows. Fix an integer $r \geq 1$ such that $r K_F$
is effective. Suppose that $m_0 K_X - \fu H$ is pseudo-effective for some $m_0 \geq
1$. Then $L = (k+\ell+1)(m_0 K_X - \fu H)$ is of course also pseudo-effective for
$k,\ell \geq 1$, and by putting things together in the right way, we get
\begin{equation} \label{eq:factors}
	\fl \OX \bigl( m r K_X - \fu H \bigr) \cong 
	\fl \OX \bigl( K_X + L_n \bigr) \tensor \OY(kH) \otimes \OY(n K_Y + \ell H),
\end{equation}
where $n = mr - (k+\ell+1)m_0 - 1$ and $L_n = nK_{X/Y} + L$. Since $mr K_F$ is effective,
\[
	\fl \OX \bigl( K_X + L_n \bigr) = \fl \OX \bigl( K_X + n K_{X/Y} + L \bigr)
\]
is a torsion-free sheaf on $Y$, of generic rank $\dim H^0(F, mr K_F)$. By the work of
P\u{a}un and Takayama \cite{PT}, the line bundle $L_n$ has a naturally defined
semi-positively curved singular hermitian metric $h_n$, and for $m \gg 0$, the inclusion
\[
	\fl \bigl( \OX(K_X + L_n) \tensor \shI(h_n) \bigr) \subseteq \fl \OX(K_X + L_n)
\]
becomes generically an isomorphism. We can then use the Koll\'ar-type vanishing
theorem for the sheaf $\fl \bigl( \OX(K_X + L_n) \tensor \shI(h_n) \bigr)$, proved by
Fujino and Matsumura \cite{FM}, to get the desired conclusion.

\newpar
Now we turn to the details. Let us first deal with the factor $\OY(nK_Y
+ \ell H)$ in \eqref{eq:factors}. Since $K_Y$ is pseudo-effective, we can choose
$\ell \geq 1$ in such a
way that the divisor $nK_Y + \ell H$ is effective for every $n \geq 1$; this kind of
result is sometimes called ``effective non-vanishing'', and is a simple consequence
of the Riemann-Roch theorem and the Kawamata-Viehweg vanishing theorem
\cite[11.2.14]{Lazarsfeld}.

\newpar
To deal with the remaining factors, we need a version of effective non-vanishing for
direct images of adjoint bundles. Here is the precise statement.

\begin{lemma} \label{lem:effective-nonvanishing}
	Let $f \colon X \to Y$ be an algebraic fiber space. Let $H$ be an ample divisor on
	$Y$. Then there is an integer $k \geq 1$ such that
	\[
		H^0 \bigl( Y, \fl \OX(K_X + L)  \tensor \OY(kH) \bigr) \neq 0
	\]
	for every line bundle $(L, h_L)$ with a semi-positively curved singular hermitian
	metric on $X$ such that $\fl \bigl( \OX(K_X + L) \tensor \shI(h_L) \bigr) \neq 0$.
\end{lemma}

\begin{proof}
	After replacing $H$ by a sufficiently big multiple, we can assume that $H$ is
	effective to begin with. To shorten the formulas, define
	\[
		\shF = \fl \bigl( \OX(K_X + L) \tensor \shI(h_L) \bigr).
	\]
	This is a nontrivial torsion-free coherent sheaf on $Y$. According to the
	Koll\'ar-type vanishing theorem proved by Fujino and Matsumura
	\cite[Thm.~D]{FM}, we have
	\[
		H^i \bigl( Y, \shF \tensor \OY(kH) \bigr) = 0
	\]
	for every $i,k \geq 1$. Consequently, 
	\[
		\dim H^0 \bigl( Y, \shF \tensor \OY(kH) \bigr) = 
		\chi \bigl( Y, \shF \tensor \OY(kH) \bigr)
	\]
	is a nonzero polynomial in $k$, of degree at most $\dim Y$, hence must be nonzero for
	at least one value of $k \in \{1, 2, \dotsc, \dim Y + 1\}$. In fact, because $H$ is
	effective, the value $k = \dim Y + 1$ always works. Now it remains to observe that
	\[
		H^0 \bigl( Y, \shF \tensor \OY(kH) \bigr) \subseteq
		H^0 \bigl( Y, \fl \OX(K_X + L) \tensor \OY(kH) \bigr),
	\]
	due to the inclusion $\shI(h_L) \subseteq \OX$.
\end{proof}

\newpar
We return to our problem. Fix integers $m_0, k, \ell \geq 1$ as above. The line
bundle $L = (k+\ell+1)(m_0 K_X - \fu H)$ is pseudo-effective,
and so it has a singular hermitian metric $h$ with semi-positive curvature
\cite[Prop.~6.6]{Demailly}. 
Denote by $(L_F, h_F)$ the restriction of $(L,h)$ to the general fiber $F$ of the
algebraic fiber space $f \colon X \to Y$.

\begin{lemma}
	For $n \gg 0$, we have $\shI(h_F^{1/n}) = \shO_F$.
\end{lemma}

\begin{proof}
	Since $F$ is compact, we can cover $F$ by finitely many open sets on which $L_F$ is
	trivial. In any trivialization, the singular hermitian metric $h_F$ can be written
	in the form $e^{-\varphi}$, with $\varphi$ plurisubharmonic. This reduces the
	problem to the following local statement: Let $\varphi$
	be a plurisubharmonic function on a neighborhood of $0 \in \CC^d$, not
	identically equal to $-\infty$; then the function $e^{-\varphi/n}$
	is locally integrable for $n \gg 0$. This is straightforward. The Lelong
	number $\nu(\varphi,0)$ is finite by construction; see for example
	\cite[Thm.~2.8]{Demailly}. But by a result of Skoda \cite[Lem.~5.6]{Demailly}, the
	function $e^{-\varphi/n}$ becomes locally integrable once $2n > \nu(\varphi,0)$.
\end{proof}

\newpar \label{par:PT}
Choose $m$ sufficiently large so that $\shI(h_F^{1/n}) = \shO_F$; recall from the
construction above that $n = rm - (k+\ell+1)m_0 -1$. Since we know that
\[
	\fl \OX \bigl( K_X + L_n \bigr) = \fl \OX \bigl( K_X + n K_{X/Y} + L \bigr) 
\]
is nontrivial, we can now apply \cite[Thm.~5.1.2]{PT}. During the proof of this
result, P\u{a}un and Takayama show that the twisted Narasimhan-Simha metric on the
fibers of $f \colon X \to Y$ induces a semi-positively curved singular hermitian
metric $h_n$ on the line bundle $L_n = n K_{X/Y} + L$, and that the inclusion
\[
	\fl \bigl( \OX(K_X + L_n) \tensor \shI(h_n) \bigr)
		\subseteq \fl \OX(K_X + L_n)
\]
is generically an isomorphism; this is a consequence of the fact that
$\shI(h_F^{1/n}) = \shO_F$. Together with \lemmaref{lem:effective-nonvanishing},
these two facts imply that
\[
	H^0 \bigl( Y, \fl \OX(K_X + L_n) \tensor \OY(kH) \bigr) \neq 0.
\]
Since $nK_Y + \ell H$ is effective, we conclude from \eqref{eq:factors} that
$\fl \OX \bigl( mr K_X - \fu H \bigr)$ has nontrivial global sections for $m
\gg 0$, and hence that $mr K_X - \fu H$ becomes effective for $m \gg 0$. This is what
we wanted to show.

\newpar
Here is an example where this gives a new result. 

\begin{example}
	Let $f \colon X \to A$ be an algebraic fiber space over an abelian variety, with
	$\kappa(F) \geq 0$. If $m K_X - \fu H$ is pseudo-effective for some $m \geq 1$ and
	some ample divisor $H$ on the abelian variety, then $m K_X - \fu H$ is effective
	for $m \gg 0$, and therefore $\kappa(X) = \kappa(F) + \dim A$. According to
	\cite[Thm.~2.1]{one-forms}, it also follows that the pullback of
	every holomorphic one-form on $A$ has a non-empty zero locus on $X$, but now under
	the much weaker assumption that $m K_X - \fu H$ is pseudo-effective (instead of
	effective).
\end{example}

\newpar
What about the case when $K_Y$ is not pseudo-effective? By \cite[Cor.~0.3]{BDPP},
this is equivalent to $Y$ being covered by rational curves. Using the MRC fibration
and \theoremref{thm:main}, one can easily
reduce the proof of \conjectureref{conj:CP-strong} to the case where $Y$ is
rationally connected. The argument is very similar to the proof of
\propositionref{prop:Iitaka}, and so we only sketch it. After a birational
modification, we can assume that the MRC fibration of $Y$ is represented by an
algebraic fiber space $p \colon Y \to Z$. If we denote by $P$ the general fiber of
$p$, then $P$ is rationally connected \cite[IV.5]{Kollar}. On the other hand, $Z$ is
not uniruled \cite[Cor.~1.4]{GHS}, and so $K_Z$ is pseudo-effective. Set $g = p \circ
f$, and again denote by $G$ the general fiber of the resulting algebraic fiber space
$g \colon X \to Z$. 
\[
\begin{tikzcd}
X \arrow[bend right=40]{rr}{g} \rar{f} & Y \rar{p} & Z
\end{tikzcd}
\]
We can assume that the ample line bundle has the form $L = L' + \pu L''$,
with $L'$ relatively ample over $Z$, and $L''$ ample on $Z$. Now $f \colon G \to P$ is an
algebraic fiber space over a rationally
connected base, with general fiber $F$. Since $mK_X - \fu L$ is pseudo-effective,
$mK_G - \fu (L' \restr{P})$ is also pseudo-effective; if \conjectureref{conj:CP-strong}
is true for the algebraic fiber space $f \colon G \to P$, then we get 
\[
	\kappa(G) = \kappa(F) + \dim P.
\]
At the same time, \theoremref{thm:main} applies to the algebraic fiber space
$g \colon X \to Z$, because $K_Z$ is pseudo-effective; together with the identity
above, this gives 
\[
	\kappa(X) = \kappa(G) + \dim Z = \kappa(F) + \dim Y.
\]
As explained in \parref{par:Mori}, this identity is equivalent to $mK_X - \fu L$ being
effective for $m$ sufficiently large and divisible. In this way,
\conjectureref{conj:CP-strong} is reduced to the case where $Y$ is
rationally connected.

\newpar
Unfortunately, I am not able to say anything even in the case $Y = \PP^1$.

\begin{example}
	Let $f \colon X \to \PP^1$ be an algebraic fiber space with $\kappa(F) \geq 0$. If
	the divisor class $m K_X - \fu \shO(1)$ is pseudo-effective for some $m \geq 1$,
	then is it true that $m K_X - \fu \shO(1)$ must be effective for $m \gg 0$? The
	argument from above breaks down because the canonical bundle of the base is no longer
	pseudo-effective.
\end{example}

\newpar
One can imagine a proof of \conjectureref{conj:CP-strong} along the following
lines. Choose $k,m \geq 1$ such that $mK_X - \fu(kH)$ is pseudo-effective and
$H^0(F, K_F + mK_F) \neq 0$. Suppose that one could find a singular hermitian metric
$h$ on the line bundle $m K_X - \fu(kH)$, with semi-positive curvature, such
that all sections in $H^0(F, K_F + m K_F)$ are square-integrable relative to
the induced metric $h_F$ on $m K_F$. (In fact, just one nontrivial section would be
enough.) Then the direct image sheaf
\[
	\shF = \fl \Bigl( \OX \bigl( K_X + mK_X - \fu(kH) \bigr) \tensor \shI(h) \Bigr)
\]
would be nonzero, and one could get the desired result by adjusting the coefficients
$k$ and $m$, and repeating the argument from above. The construction in
\parref{par:PT} actually produces such a metric when $K_Y$ is pseudo-effective. Does the
same kind of metric still exist when $K_Y$ is not pseudo-effective?

\providecommand{\bysame}{\leavevmode\hbox to3em{\hrulefill}\thinspace}
\providecommand{\MR}{\relax\ifhmode\unskip\space\fi MR }
\providecommand{\MRhref}[2]{%
  \href{http://www.ams.org/mathscinet-getitem?mr=#1}{#2}
}
\providecommand{\href}[2]{#2}

\end{document}